\documentclass[xcolor=dvipsnames,svgnames,table,reqno]{amsart}

\input xy
\xyoption{all}
\usepackage{accents}
\usepackage{epsfig}
\usepackage{xcolor}
\usepackage{amsthm}
\usepackage{amssymb}
\usepackage{bbm}
\usepackage{amsmath}
\usepackage{amscd}
\usepackage{amsopn}
\usepackage{graphicx}
\usepackage{tikz-cd} 
\usepackage{xspace}

\usepackage{hhline}
\usepackage{easybmat}
\usepackage{caption}   
\usepackage{relsize}

\usepackage{url}
\usepackage{enumitem, hyperref}\hypersetup{colorlinks}

\usepackage{stmaryrd}

\usepackage[T1]{fontenc}

\colorlet{purpleB70}{blue!70!red}

\colorlet{orangeR65}{red!65!yellow}

\definecolor{red2}{HTML}{d41173}

\definecolor{neongreen}{HTML}{1bf702}

\definecolor{radicalred}{HTML}{FF355E}

\definecolor{denim}{HTML}{1560BD}

\definecolor{darkcyan}{rgb}{0.0, 0.55, 0.55}

\definecolor{cilek}{HTML}{FF43A4}

\definecolor{mor}{HTML}{9F00C5}

\definecolor{phlox}{rgb}{0.87, 0.0, 1.0}

\definecolor{fluorescentpink}{HTML}{FF1493}

\definecolor{napiergreen}{rgb}{0.16, 0.5, 0.0}

\definecolor{kellygreen}{rgb}{0.3, 0.73, 0.09}

\definecolor{parisgreen}{HTML}{ 50C878 }

\definecolor{palatinateblue}{rgb}{0.15, 0.23, 0.89}

\definecolor{ceruleanblue}{rgb}{0.16, 0.32, 0.75}

\definecolor{brandeisblue}{rgb}{0.0, 0.44, 1.0}

\definecolor{KLMblue}{HTML}{0FC0FC}

\definecolor{cinnamon}{rgb}{0.82, 0.41, 0.12}

\definecolor{darkorange}{rgb}{1.0, 0.55, 0.0}

\definecolor{darktangerine}{rgb}{1.0, 0.66, 0.07}

\definecolor{deepcarrotorange}{rgb}{0.91, 0.41, 0.17}

\definecolor{internationalorange}{HTML}{FF4F00}

\definecolor{persimmon}{HTML}{EC5800}

\definecolor{pumpkin}{HTML}{FF7518}

\definecolor{darkred}{rgb}{1,0,0} 
\definecolor{darkgreen}{rgb}{0,0.7,0}
\definecolor{darkblue}{rgb}{0,0,1}

\hypersetup{colorlinks,
linkcolor=darkblue,
filecolor=darkgreen,
urlcolor=darkred,
citecolor=darkgreen}

\makeatletter
\def\reflb#1#2{\begingroup
    #2%
    \def\@currentlabel{#2}%
    \phantomsection\label{#1}\endgroup
}
\makeatother

\numberwithin{equation}{section}
\newtheorem{Theorem}{Theorem}
\numberwithin{Theorem}{section}

\newtheorem   {Lemma}[Theorem]{Lemma}

\newtheorem   {Proposition}[Theorem]{Proposition}
\newtheorem   {Corollary}[Theorem]{Corollary}
\theoremstyle {definition}

\theoremstyle {remark}
\newtheorem   {Remark}[Theorem]{Remark}
\newtheorem   {Example}[Theorem]{Example}

\def    \eps    {\epsilon}

\newcommand{\CC}{{\mathcal C}}

\newcommand{\CL}{{\mathcal L}}

\newcommand{\CT}{{\mathcal T}}

\newcommand{\supp}{\operatorname{supp}}

\newcommand{\id}{{\mathit id}}

\newcommand{\fg}{{\mathfrak g}}

\newcommand{\fd}{{\mathfrak d}}

\newcommand{\tN}{\tilde{N}}

\newcommand{\tL}{\tilde{L}}

\newcommand{\CB}{{\mathcal B}}

\def    \F      {{\mathbb F}}

\def    \C      {{\mathbb C}}
\def    \R      {{\mathbb R}}

\def    \Z      {{\mathbb Z}}

\def    \T      {{\mathbb T}}
\def    \CP     {{\mathbb C}{\mathbb P}}
\def    \RP     {{\mathbb R}{\mathbb P}}

\def    \12     {{\frac{1}{2}}}

\def    \p      {\partial}
\def    \codim  {\operatorname{codim}}

\def    \HF     {\operatorname{HF}}

\def    \H      {\operatorname{H}}

\def    \CF      {\operatorname{CF}}

\def    \Stab     {\operatorname{Stab}}

\def    \vol     {\mathit{vol}}

\def    \s     {\operatorname{c}}

\def    \inv   {\mathrm{inv}}

\begin{document}


\setlength{\smallskipamount}{6pt}
\setlength{\medskipamount}{10pt}
\setlength{\bigskipamount}{16pt}





\title [Lower Semi-continuity of Lagrangian Volume]{Lower
  Semi-continuity of Lagrangian Volume}

\author[Erman \c C\. inel\. i]{Erman \c C\. inel\. i}
\author[Viktor Ginzburg]{Viktor L. Ginzburg}
\author[Ba\c sak G\"urel]{Ba\c sak Z. G\"urel}

\address{E\c C: Institut de Math\'ematiques de Jussieu - Paris Rive
  Gauche (IMJ-PRG), 4 place Jussieu, Boite Courrier 247, 75252 Paris
  Cedex 5, France} \email{erman.cineli@imj-prg.fr}

\address{VG: Department of Mathematics, UC Santa Cruz, Santa
  Cruz, CA 95064, USA} \email{ginzburg@ucsc.edu}

\address{BG: Department of Mathematics, University of Central Florida,
  Orlando, FL 32816, USA} \email{basak.gurel@ucf.edu}



\date{\today} 

\thanks{The work is partially supported by NSF CAREER award
  DMS-1454342 (BG), Simons Foundation Collaboration Grants 581382 (VG)
  and 855299 (BG) and ERC Starting Grant 851701 via a postdoctoral
  fellowship (E\c{C})}


\begin{abstract}
  We study lower semi-continuity properties of the volume, i.e., the
  surface area, of a closed Lagrangian manifold with respect to the
  Hofer- and $\gamma$-distance on a class of monotone Lagrangian
  submanifolds Hamiltonian isotopic to each other. We prove that
  volume is $\gamma$-lower semi-continuous in two cases. In the first
  one the volume form comes from a K\"ahler metric with a large group
  of Hamiltonian isometries, but there are no additional constraints
  on the Lagrangian submanifold. The second one is when the volume is
  taken with respect to any compatible metric, but the Lagrangian
  submanifold must be a torus. As a consequence, in both cases, the
  volume is Hofer lower semi-continuous.
\end{abstract}

\maketitle

\vspace{-0.2in}

\tableofcontents

\section{Introduction and main results}
\label{sec:intro+results}

\subsection{Introduction}
\label{sec:intro}
In this paper we are concerned with lower semi-continuity properties
of the volume, i.e., the surface area, of a closed Lagrangian manifold
with respect to the distance of a purely symplectic topological
nature, e.g., the Hofer- and $\gamma$-distance, on a class of monotone
Lagrangian manifolds Hamiltonian isotopic to each other.

We conjecture that volume is $\gamma$-lower semi-continuous in
general, and we prove this in two situations. The first one is fairly
close to the standard setting of integral geometry. This is the case
where the volume form comes from a K\"ahler metric with a very large
group of Hamiltonian isometries, but there are no additional
constraints on the Lagrangian submanifold. The second one is in some
sense much more general: the volume form is taken with respect to any
compatible metric, but the Lagrangian submanifold must be a torus. As
a consequence, in both cases, the volume is lower semi-continuous with
respect to the Hofer metric.

The question is inspired by the key result from \cite{AM21} asserting
that in dimension two the topological entropy of a Hamiltonian
diffeomorphism is Hofer lower semi-continuous. We find results of this
type quite interesting because they connect seemingly unrelated
entities existing in completely different realms: pure dynamics or
metric invariants such as topological entropy or volume on one side
and symplectic topological features on the other.

The second motivation for the question comes from \cite{CGG:Entropy}
where topological entropy of compactly supported Hamiltonian
diffeomorphisms is connected with Hamiltonian or Lagrangian Floer
theory via the so-called barcode entropy which is determined by the
growth of the number of not-too-short bars in the filtered Floer
complex of the iterates. That paper also provides a natural framework
to study the question by connecting Crofton's type (in)equalities from
integral geometry with Floer theory. Here we use the notion of
Lagrangian tomograph introduced in that paper to show that for a large
class of $n$-densities on a $2n$-dimensional symplectic manifold the
integral over a Lagrangian submanifold is $\gamma$-lower
semi-continuous. Then these densities are used to match or at least
approximate from below the metric $n$-density.

One can pose a similar question about other metric (or dynamics)
invariants, but the choice of volume is quite natural for it already
enjoys strong $C^0$-lower semi-continuity properties; see \cite{Ce,
  Fe52} and also \cite{BuIv, Iv}. It is then only reasonable to ask if
there is an analogue in the symplectic setting. In some situations the
$\gamma$-norm is known to be continuous with respect to the
$C^0$-topology (see \cite{BHS, KS}), and hence, at least on the
conceptual level, $\gamma$-lower semi-continuity is a refinement of
$C^0$-lower semi-continuity in the symplectic framework.

Another interpretation of our results is that the volume function
extends to a lower semi-continuous function on the Humili\`ere
completion, i.e., the completion with respect to the $\gamma$-distance
(see \cite{Hu}), of the class of Lagrangians Hamiltonian isotopic to
each other, whenever $\gamma$-lower semi-continuity is established.

\medskip\noindent{\bf Acknowledgements.} Parts of this work were
carried out while the second and third authors were visiting the
IMJ-PRG, Paris, France, in May 2022 and also during the
\emph{Symplectic Dynamics Beyond Periodic Orbits} Workshop at the
Lorentz Center, Leiden, the Netherlands in August 2022. The authors
would like to thank these institutes for their warm hospitality and
support.

\subsection{Main results}
\label{sec:results}
Let $(M^{2n},\omega)$ be a monotone symplectic manifold which is
either closed or sufficiently nice at infinity (e.g., convex) to
ensure that the relevant filtered Floer homology is defined; see
Section \ref{sec:conv} and Remark \ref{rmk:infty}. Furthermore, let
$\CL$ be a class of closed monotone Lagrangian submanifolds $L$ of
$M$, Hamiltonian isotopic to each other. We require in addition that
the minimal Chern number of $L$ is at least 2.

Recall that the $\gamma$-norm of a compactly supported Hamiltonian
diffeomorphism $\varphi$ is
$$
\gamma(\varphi):=\inf_H\big(\s(H)+\s(H^{inv})\big),
$$
where the infimum is taken over all compactly supported Hamiltonians
$H$ generating $\varphi$ as the time-one map $\varphi_H$ of the
Hamiltonian isotopy $\varphi_H^t$, the Hamiltonian $H^{\inv}$
generates the isotopy $(\varphi_H^t)^{-1}$ and $\s$ is the spectral
invariant associated with the fundamental class $[M]$ (relative
infinity when $M$ is not compact); see \cite{Oh:spec,Sc,Vi:gen}.

The \emph{(ambient) $\gamma$-distance} on $\CL$ is defined as
$$
d_\gamma(L,L'):=\inf\{\gamma(\varphi)\mid \varphi(L)=L'\}.
$$
This is indeed a distance on $\CL$; see, e.g., \cite{KS} and
references therein.

\begin{Example}
  \label{ex:n=1}
  Let $L$ and $L'$ be Hamiltonian isotopic loops on a surface
  $M$. Then $d_\gamma(L,L')$ is the total area displaced by a
  Hamiltonian isotopy from $L$ to $L'$. (We do not intend here to make
  this notion precise but rather rely on geometric intuition.) This
  shows that two loops $L$ and $L'$ which are $d_\gamma$-close need
  not be close with respect to the Hausdorff distance. This is the
  case for instance when $L'$ is obtained from $L$ by growing long but
  narrow tongues (or tentacles) and, if needed, by a small
  perturbation to keep $L$ and $L'$ Hamiltonian isotopic.
\end{Example}

Alternatively, when the Lagrangian submanifolds $L$ from $\CL$ are
wide in the sense of \cite{BC}, i.e., $\HF(L)=\H(L)\otimes\Lambda$,
where $\Lambda$ is the Novikov ring, one has the \emph{(interior)
  $\gamma$-distance}. It is defined in a similar fashion but now by
using Lagrangian spectral invariants; see, e.g.,
\cite{Le,LZ,KS,Vi:gen} for further details and references. Among wide
Lagrangian submanifolds are the zero section of a cotangent bundle and
the ``equator'' $\RP^n\subset \CP^n$. On the other hand, displaceable
Lagrangian submanifolds have $\HF(L)=0$ and hence are not wide.

In general, the interior $\gamma$-distance is bounded from above by
the ambient $\gamma$-distance which in turn is bounded from above by
the Hofer distance. We are not aware of any example where the two
$\gamma$-distances are different. Our results hold for both the
ambient $\gamma$-distance and the interior $\gamma$-distance, when the
latter is defined. We will not distinguish the two distances and will
use the same notation $d_\gamma$.

Finally, fix a Riemannian metric compatible with $\omega$. Then we
have the volume or, to be more precise, the surface area function:
$$
\vol\colon \CL\to (0,\infty)
$$
sending $L$ to its surface area, which we refer to as the Lagrangian
volume.

We conjecture that $\vol$ is lower semi-continuous on $\CL$ with
respect to the $\gamma$-distance, and here we prove this conjecture in
two disparate cases. The first of these is where $M$ is K\"ahler and
has a large symmetry group (e.g., $M=\C^n$ or $\CP^n$).

\begin{Theorem}
  \label{thm:symmetry}
  Let $M$ be K\"ahler and the Riemannian metric be the real part of
  the K\"ahler form. Assume furthermore that the group of Hamiltonian
  K\"ahler isometries acts transitively on the Lagrangian Grassmannian
  bundle over $M$. Then $\vol$ is lower semi-continuous on $\CL$ with
  respect to the $\gamma$-distance.
\end{Theorem} 

This theorem is proved in Section \ref{sec:tomographs-barcodes}.

When $L=\T^n$, the restrictive condition that $M$ has a large symmetry
group can be dropped. In fact, we have a more precise result asserting
roughly speaking that for a fixed Lagrangian submanifold $L_0$ from
$\CL$ and another Lagrangian submanifold $L\in \CL$, which is
$d_\gamma$-close to $L_0$ (depending on $L_0$), the part of $L$
situated $C^0$-close to $L_0$ is at least almost as large as $L_0$.

\begin{Theorem}
  \label{thm:torus}
  Assume that $L_0=\T^n\in\CL$ and let $U$ be an arbitrary open subset
  containing $L_0$. Then the function
$$
\vol_U\colon \CL\to [0,\infty)
$$
sending $L$ to the surface area of $U\cap L$ is lower semi-continuous
on $\CL$ at $L_0$ with respect to the $\gamma$-distance.
\end{Theorem} 

The proofs of Theorems \ref{thm:symmetry} and \ref{thm:torus} rely on
a result of independent interest, Theorem \ref{thm:density}, asserting
$d_\gamma$-lower semi-continuity of the integral of certain densities
and based on a connection between Floer barcodes and Lagrangian
tomographs; cf.\ \cite{CGG:Entropy,CGG:Growth}.

In Theorem \ref{thm:torus}, $U$ and $L_0$ are tied up by the
requirement that $L_0\subset U$. Although we do not have a proof of
this, we expect this requirement to be unnecessary, i.e., that the
function $\vol_U$ is lower semi-continuous at every point of $\CL$ for
any open set $U\subset M$. In any event, as an immediate consequence
of Theorem \ref{thm:torus}, we have

\begin{Corollary}
  \label{cor:torus}
  Assume that $L=\T^n$. Then $\vol$ is lower semi-continuous on $\CL$
  with respect to the $\gamma$-distance.
\end{Corollary}

Note that since the $\gamma$-distance is bounded from above by the
Hofer distance, in the setting of this corollary or of Theorem
\ref{thm:symmetry}, the $\vol$ function is also lower semi-continuous
with respect to the Hofer distance.

\begin{Remark}
  Another consequence of Theorem \ref{thm:torus} is that
  $L\cap U\neq \emptyset$. Here, however, a much more precise and
  general result is available, which, in particular, does not require
  $L$ to be a torus. Namely, when $L$ is $d_\gamma$-close to $L_0$,
  for every point of $L_0$ the submanifold $L$ intersects a small ball
  centered at that point. This is an immediate consequence of, for
  example, \cite[Thm.\ F]{KS}; see also, e.g., \cite{BaC,BC,Vi:sup}
  for some relevant results.
\end{Remark}

\begin{Example}
  \label{ex:min}
  The function $\vol$ is automatically lower semi-continuous at $L$
  when $L\subset M$ is a local or global volume minimizer in
  $\CL$. For instance, as is easy to see, this is the case for the zero
  section of the cotangent bundle equipped with the Sasaki
  metric. Likewise, the standard $\RP^n\subset \CP^n$ and the Clifford
  torus are volume minimizers with respect to the Fubini--Studi metric
  on $\CP^n$; \cite{Oh:Invent}. The same is true for the product tori in
  $\C^n$ with respect to the standard metric; \cite{Oh:MathZ}. In
  contrast, Theorem \ref{thm:symmetry} and Corollary \ref{cor:torus}
  assert lower semi-continuity at every point of $\CL$ and in the case
  of the corollary for a broad class of metrics on $M$.
\end{Example}

\begin{Remark}
  \label{rmk:n=1}
  When $n=1$ as in Example \ref{ex:n=1}, Corollary \ref{cor:torus}
  asserts that the length of an embedded loop is lower semi-continuous
  under deformations of the loop preserving the area bounded by the
  loop with respect to the displaced area taken as a metric. (Such
  deformations are defined even when the loop does not bound a
  domain.)  This fact must be known in some form but we are not aware
  of any reference.
\end{Remark}

Denote by $\hat{\CL}$ the Humili\`ere completion of $\CL$, i.e., its
completion with respect to the $\gamma$-distance; cf., \cite{Hu}. The
Corollary \ref{cor:torus} and Theorem \ref{thm:symmetry} are
equivalent to the following result.

\begin{Corollary}
  \label{cor:compl}
  Assume that $L=\T^n$ or that $M$ is as in Theorem
  \ref{thm:symmetry}. Then $\vol$ extends to a lower semi-continuous
  function on $\hat{\CL}$.
\end{Corollary}

\begin{Remark}
  We do not know if in general the function $\vol$ is bounded away
  from zero on $\CL$ or equivalently on $\hat{\CL}$. This is obviously
  so when $\vol$ has a global minimizer as in the setting of Example
  \ref{ex:min}. Furthermore, lower bounds for $\vol(L)$ in terms of
  the displacement energy of $L$ are obtained in \cite{Vi:metric} when
  $M=\R^{2n}$ or $\CP^n$ or a cotangent bundle.
\end{Remark}

\section{Preliminaries}
\label{sec:prelim}

\subsection{Notation and conventions}
\label{sec:conv}
Throughout this paper we use conventions and notation from
\cite{CGG:Entropy} and \cite{CGG:Growth}. Referring the reader to
\cite[Sect.\ 3]{CGG:Entropy} for a much more detailed discussion, here
we only touch upon several key points.

All Lagrangian submanifolds $L$ are assumed to be closed and monotone,
and in addition we require that the minimal Chern number of $L$ is at
least 2.  The ambient symplectic manifold $M$ is also assumed to be
monotone but not necessarily compact. In the latter case, we assume
that $M$ is sufficiently well-behaved at infinity (e.g., convex) so
that the filtered Floer complex and homology can be defined for the
pair $(L, L')$ of Hamiltonian isotopic Lagrangians; see Remark
\ref{rmk:infty} below and \cite[Rmk.\ 2.8]{CGG:Entropy} for more
details.

For the sake of simplicity Floer complexes and homology and also the
ordinary homology are taken over the ground field $\F=\F_2$. When $L$
and $L'$ are Hamiltonian isotopic and intersect transversely, we
denote by $\CF(L,L')$ the Floer complex of the pair $(L,L')$. This
complex is generated by the intersections $L\cap L'$ over the
\emph{universal Novikov field} $\Lambda$. This is the field of formal
sums
$$
\lambda=\sum_{j\geq 0} f_j T^{a_j},
$$
where $f_j\in \F$ and $a_j\in\R$ and the sequence $a_j$ (with
$f_j\neq 0$) is either finite or $a_j\to\infty$.

Due to our choice of the Novikov field, the complex $\CF(L,L')$ is not
graded. However, fixing a Hamiltonian isotopy from $L$ to $L'$ and
``cappings'' of intersections we obtain a filtration on $\CF(L,L')$ by
the Hamiltonian action. The differential on the complex is defined in
the standard way.

Note that the complex breaks down into a direct sum of subcomplexes
over homotopy classes of paths from $L$ to $L'$. Then to define the
action filtration on $\CF(L,L')$ we also need to pick a reference path
in every homotopy class.

To an $\R$-filtered, finite dimensional complex $\CC$ over $\Lambda$
one then associates its barcode $\CB$. In the most refined form this
is a collection of finite or semi-infinite intervals, defined in
general up to some shift ambiguity. The number of semi-infinite
intervals is equal to $\dim_{\Lambda}\H(\CC)$. A construction of $\CB$
most suitable for our purposes is worked out in detail in \cite{UZ}
and also briefly discussed in \cite{CGG:Entropy}.  For our goals, it
is convenient to forgo the location of the intervals and treat $\CB$
as a collection (i.e., a multiset) of positive numbers including
$\infty$. Setting $\CC=\CF(L,L')$ we obtain the barcode
$\CB(L,L')$. With this convention the barcode $\CB(L,L')$ is
independent of the choices involved in the definition of the action
filtration.

The actual definition of $\CB$ is not essential for our purposes and
its only feature that matters is that it is continuous in the
Hamiltonian or the Lagrangian with respect to the $C^\infty$-topology
and even the Hofer norm or the $\gamma$-norm.  To be more precise,
denote by $b_\eps=b_\eps(L,L')$ the number of bars in the barcode of
length greater than $\eps$. This is the main ingredient in the
definition of barcode entropy; see \cite{CGG:Entropy}.  Assume
furthermore that Lagrangian submanifolds $L$, $L'$ and $L''$ are
Hamiltonian isotopic, and $d_{\gamma}(L',L'')<\delta/2$, and $L'$ and
$L''$ are transverse to $L$. Then
\begin{equation}
  \label{eq:b-gamma}
  b_{\eps}(L,L'')\geq b_{\eps+\delta}(L,L').
\end{equation}
This is a consequence of \cite[Thm.\ G]{KS}; see also \cite{Vi:red}.
This property allows one to extend the definition of the barcode and
of $b_\eps$ ``by continuity'' to the case where the manifolds are not
transverse.

Furthermore, assuming that $L\pitchfork L'$ note that
\begin{equation}
  \label{eq:N-b1} 
\dim_\Lambda \CF(L,L')\geq 2b_\eps(L,L')-\dim_\Lambda \HF(L).
\end{equation}
This inequality turns into an equality when $\eps$ is smaller than the
shortest bar, i.e., $b_\eps$ is the total number of bars $b(L,L')$:
\begin{equation}
\label{eq:N-b2}
\dim_\Lambda \CF(L,L')= 2b(L,L')-\dim_\Lambda \HF(L).
\end{equation}

\begin{Remark}[Conditions on $M$ at infinity]
  \label{rmk:infty}
  In this remark we touch upon the conditions, in addition to being
  monotone, that $M$ must satisfy at infinity when it is not
  compact. If $L$ is wide in the sense of \cite{BC}, i.e.,
  $\HF(L)=\H(L)\otimes \Lambda$, we can work with the interior
  $\gamma$-norm and it is sufficient to assume that $M$ is
  geometrically bounded. Otherwise, we use the ambient
  $\gamma$-norm. In this case we need to have the filtered Floer
  homology and the fundamental class spectral invariant defined for
  compactly supported Hamiltonians $H$ on $M$. To this end, we can
  require $M$ to be geometrically bounded and wide in the sense of
  \cite{Gu}, i.e., admitting a proper function
  $F\colon M\to [0,\infty)$ without non-trivial contractible periodic
  orbits of period less than or equal to one. Indeed, such a function
  can then be found vanishing on any compact set and used to perturb
  $H$ at infinity. The resulting Floer (co)homology is well-defined
  and isomorphic to $\H^*_c(M)\otimes \Lambda$. Hence the required
  spectral invariant is also defined. Alternatively, we may require
  $M$ to be convex at infinity; \cite{FS}.
\end{Remark}  

\subsection{Input from integral geometry}
\subsubsection{Densities}
\label{sec:densities}
Let $P$ be the Stiefel bundle over a manifold $M^m$, i.e., $P$ is
formed by $k$-frames $\bar{v}=(v_1,\ldots,v_k)$.  Recall that a
$k$-density $\fd$ on $M$ is a function $\fd\colon P\to \R$ such that
\begin{equation}
  \label{eq:density}
\fd(\bar{v}')=|\det A|\fd(\bar{v}),
\end{equation}
where $A$ is the linear transformation of the span of $\bar{v}$
sending $\bar{v}$ to $\bar{v}'$; see, e.g., \cite{APF98}. Sometimes it
is convenient to drop the condition that the vectors from
$\bar{v}=(v_1,\ldots, v_k)$ are linearly independent by setting
$\fd(\bar{v})=0$ otherwise.

Here are several examples of densities: A Riemannian or Finsler metric
on $M$ or, more generally, any homogeneous degree-one function
$TM\to\R$ is a one-density. For instance, in self-explanatory
notation, the functions $|dx|$, $|dy|$, $|dx|+|dy|$ and
$\sqrt{dx^2+dy^2}$ are 1-densities on $\R^2$.  Furthermore, for every
$k\leq m$ a Riemannian metric gives rise to a $k$-density $\fg_k$
defined by the condition that $\fg_k(\bar{v})$ is the volume of the
parallelepiped spanned by $\bar{v}$. Thus $\fg_m$ is the Riemannian
volume. For a differential $k$-form $\alpha$ its absolute value
$|\alpha|$ is a $k$-density. The sum of two $k$-densities is again a
$k$-density.

A $k$-density $\fd$ can be integrated over a compact $k$-dimensional
submanifold $L$ without requiring $L$ to be oriented or even
orientable.  Similarly to differential forms, densities can be pulled
back and, under suitable additional conditions, pushed forward. When
it is defined, the push-forward $\Psi_*\fd$ of $\fd$ by a map
$\Psi\colon M'\to M$ is characterized by the condition that
$$
\int_L \Psi_*\fd=\int_{\Psi^{-1}(L)}\fd.
$$

\subsubsection{Lagrangian tomographs and Crofton's formula}
\label{sec:tomographs}
Among the key tools entering the proofs of our results are
Lagrangian tomographs. In this section we briefly discuss the notion
following with minor modifications \cite{CGG:Entropy,CGG:Growth},
which in turn is loosely based on \cite{APF98,APF07,GS} and also
\cite{GuSt,GuSt:book}.

For our purposes, a \emph{tomograph} $\CT$ comprises the following
data:
\begin{itemize}
\item a fiber
  bundle $\pi\colon E\to B$ with fiber $K$;
\item a map $\Psi\colon E\to M$, which is required to be a submersion
  onto its image and an embedding of every fiber $\pi^{-1}(s)$,
  $s\in B$;
\item a smooth measure $ds$ on $B$.
\end{itemize}
Here the fiber $K$ is required to be a closed manifold; the base $B$
may have boundary and need not be compact, but if it is not, and hence
$E$ is not compact, the submersion $\Psi$ must be proper. Finally, the
measure $ds$ is required to be supported away from $\p B$. The key
difference of this definition from the references above is that there
$\Psi$ is also a fiber bundle and hence a tomograph is a
double-fibration. (The term ``tomograph'' is not used there.)

Set $L_s:=\Psi(\pi^{-1}(s))$ and $\Psi_s:=\Psi|_{\pi^{-1}(s)}$
for $s\in B$. Then $L_s$ is a smooth closed submanifold of $M$ and
$\dim L_s=\dim K$. We call $d=\dim B$ the dimension of the
tomograph. The pull-back/push-forward density
\begin{equation}
  \label{eq:fd}
\fd_{\CT}:=\Psi_*\pi^*\, ds
\end{equation}
is a smooth $k$-density on $M$ with $k=\codim L_s$. We call $\supp
\fd_{\CT}\subset \Psi(E)$ the support of the tomograph $\CT$.

Next, let $L$ be a closed submanifold of
$M$ such that $\codim L=\dim K$. Set
$$
N(s):=|L_s\cap L|\in [0,\,\infty].
$$
Since $\Psi$ is a submersion, $\Psi_s\pitchfork L$ for almost all
$s\in B$. Hence $N(s)<\infty$ almost everywhere and $N$ is an
integrable function on $B$. We refer the reader to, e.g., \cite{APF98}
for the proof of the following simple but important result.

\begin{Proposition}[Crofton's formula; \cite{APF98}]
  \label{prop:Crofton}
  We have
  $$
  \int_B N(s)\,ds=\int_{L}\fd_{\CT}.
  $$
\end{Proposition}

\begin{Remark}
  \label{rmk:ds}
  It is useful to keep in mind that the measure $ds$ does not
  essentially enter in any of the tomograph requirements: it can be
  any smooth measure on $B$ supported away from $\p B$. The latter
  condition is imposed to ensure that $\fd_{\CT}$ is smooth. In
  particular, one can always localize the support of $\CT$ near $L_s$
  with $s$ in the interior of $B$ by localizing the support of $ds$
  near $s$.
\end{Remark}

\begin{Remark}
  This interpretation of classical Crofton's formula, which utilizes
  densities and ultimately goes back to \cite{GS}, is conceptually
  quite different from the one based on the surface area or, more
  generally, the Hausdorff measure associated with a metric as in
  \cite[Thm.\ 3.2.26]{Fe}. The latter, of course, can be applied to a
  much bigger class of subsets than submanifolds.  We also note that
  the term tomograph might be misappropriated here because our
  tomographs have a limited functionality, determining only the volume
  but not the shape of the subset; cf.\ \cite{Ga}.
\end{Remark}

Assume now that $M$ is symplectic of dimension $m=2n$. We call $\CT$ a
\emph{Lagrangian tomograph} when all submanifolds
$L_s=\Psi(\pi^{-1}(s))$ are Lagrangian and Hamiltonian isotopic to
each other. Thus a Lagrangian tomograph is a family of Lagrangian
submanifolds $L_s$ which are parametrized by $B$ and meet some
additional requirements. Note that $\dim L_s=\dim K=n=k$.

\begin{Example}[Classical Lagrangian tomographs]
  \label{exam:symmetry}
  Assume that $M$ is K\"ahler and that the group $G$ of Hamiltonian
  K\"ahler isometries acts transitively on $M$. Let $K$ be any closed
  Lagrangian submanifold of $M$. Set $E=K\times G$ with $\pi$ being
  the projection to the second factor $B=G$ and $\Psi(x,s):=s(x)$,
  where $x\in K$ and $s\in G$. Finally, we let $ds$ be a Haar measure
  on $G$. Then, as is easy to see, we obtain a Lagrangian tomograph.
  This is essentially the classical setting of Crofton's formula --
  see the reference cited above. (Usually one replaces the base $B=G$
  by the space $G/\Stab(K)$ formed by all images of $K$ in $M$ under
  $G$. Here, however, we prefer to keep $B=G$.) This construction
  applies to $M=\C^n$ and $\CP^n$ and, more generally, to any simply
  connected homogeneous K\"ahler manifold.
\end{Example}

Example \ref{exam:symmetry} produces tomographs with support equal to
$M$ and requires $M$ to have a large symmetry group. Here, as in
\cite{CGG:Entropy,CGG:Growth}, we are also interested in tomographs
supported in a small tubular neighborhood $U$ of a Lagrangian
submanifold $L\subset M$ and having $L$ as one of the submanifolds
$L_s$. (Hence, $K\cong L$.) In the setting of the example, this can be
achieved by localizing $ds$ as in Remark \ref{rmk:ds}. However, local
tomographs exist in a much more general setting. Indeed, first note
that by the Weinstein tubular neighborhood theorem we can set
$M=T^*L$. Furthermore, to construct a local tomograph near $L$ we can,
essentially without loss of generality by shrinking the support of
$ds$, assume that $E=L\times B$, where $B$ is a $d$-dimensional ball,
and $L=L_0$ is the image of the fiber over center $0\in B$. It turns
out that such localized Lagrangian tomographs always exist.

\begin{Lemma}[Lemma 5.6; \cite{CGG:Entropy}]
  \label{lemma:imm}
  A Lagrangian tomograph of dimension $d$ supported in $U$ exists if
  and only if $L$ admits an immersion into $\R^d$.
\end{Lemma}

\begin{Remark}
  In fact, the proof of the lemma shows slightly more: on the
  infinitesimal level such tomographs (without $ds$ fixed) are in
  one-to-one correspondence with immersions of $L$ into $\R^d$.
\end{Remark}

\section{Tomographs and barcodes}
\label{sec:tomographs-barcodes}

The proofs of Theorems \ref{thm:symmetry} and \ref{thm:torus} are
based on automatic $d_\gamma$-lower semi-continuity of the integral of
the pull-back/push-forward density associated with a Lagrangian
tomograph, which is in turn a consequence of a connection between
Floer barcodes and tomographs.

Thus fix a class $\CL$ of Lagrangian submanifolds of $M$ as in Section
\ref{sec:results} and let $\CT$ be a Lagrangian tomograph such that
$L_s\in\CL$. Denote by $\fd_\CT$ the pull-back/push-forward density of
$\CT$ given by \eqref{eq:fd}.

\begin{Theorem}
  \label{thm:density}
  The function
  $$
I_\CT\colon  L\mapsto \int_L\fd_{\CT}
  $$
  is $d_\gamma$-lower semi-continuous on $\CL$.
\end{Theorem}

\begin{proof}
  Fix $L\in \CL$ and $\eta>0$. Our goal is to show that
  $$
  I_\CT(\tL)\geq I_\CT(L)-\eta
  $$
  when $d_\gamma(L,\tL)$ is small.

  Let $\Sigma\subset B$ be the set of all points such that $\Psi_s$ is
  not transverse to $L$. Since $\Psi$ is a submersion, this a closed
  zero-measure subset of $B$. Thus there exists a compact subset
  $B'\subset B\setminus \Sigma$ such that in the notation from Section
  \ref{sec:tomographs}
  \begin{equation}
    \label{eq:B'}
  \int_{B'} N(s)\, ds\geq \int_{B} N(s)\, ds- \eta =I_\CT(L)-\eta,
\end{equation}
where the equality follows from Crofton's formula (Proposition
\ref{prop:Crofton}). The set $B'$ is obtained by removing from $B$ a
sufficiently small neighborhood of $\Sigma$.

By construction, $L_s\pitchfork L$ for all $s\in B'$. Since $B'$ is
compact, for all $s\in B'$ the shortest bar in the barcode of
$\CF(L,L_s)$ is bounded away from zero by some constant $\beta>0$.
Furthermore, $N(s)=\dim_\Lambda\CF(L,L_s)$.

  Take now $\eps>0$ and $\delta>0$ so small that
  $\eps+\delta<\beta$. In particular,
  $b_{\eps+\delta}(L,L_s)=b(L,L_s)$ and thus, by \eqref{eq:N-b2},
  \begin{equation}
    \label{eq:b-eps+delta}
  N(s)=2b_{\eps+\delta}(L,L_s)-h,
\end{equation}
where we set $h=\dim_\Lambda\HF(L)$ for the sake of brevity.

Assume next that $d_\gamma(L,\tL)<\delta/2$ and set
$\tN(s):=|\tL\cap L_s|$. We have $\tL \pitchfork L_s$, and hence
$\tN(s)= \dim_\Lambda\CF(\tL,L_s)$, for almost all $s\in B$. Then
\begin{equation*}
  \begin{aligned}
    I_{\CT}(\tL)  & =\int_B\tN(s)\, ds &\textrm{by Crofton's formula} \\
    & \geq \int_B \big(2b_\eps(\tL,L_s)-h\big)\, ds
    &\textrm{by \eqref{eq:N-b1}}\\
    & \geq \int_{B'} \big(2b_\eps(\tL,L_s)-h\big)\, ds &\textrm{since
      $B'\subset B$} \\
    & \geq \int_{B'} \big(2b_{\eps+\delta}(L,L_s)-h\big)\, ds
    &\textrm{by \eqref{eq:b-gamma}}\\
    & =\int_{B'} N(s)\, ds &\textrm{by \eqref{eq:b-eps+delta}}\\
    & \geq I_{\CT}(L)-\eta, &\textrm{by \eqref{eq:B'}}
\end{aligned}
\end{equation*}         
which completes the proof of the theorem.
\end{proof}

\begin{proof}[Proof of Theorem \ref{thm:symmetry}]
  In the setting of the theorem consider the Lagrangian tomograph
  $\CT$ from Example \ref{exam:symmetry}. Both the
  push-forward/pull-back density $\fd_\CT$ and the metric $n$-density
  $\fg$ are invariant under the group $G$ of (Hamiltonian) K\"ahler
  isometries. Since $G$ acts transitively on the Lagrangian
  Grassmannian bundle, these two densities agree up to a factor on the
  frames spanning Lagrangian subspaces;
  cf. \cite[Lem. 5.4]{APF07}. Hence, the function
$$
L\mapsto \int_L\fg
$$
is also $d_\gamma$-lower semi-continuous on $\CL$ by Theorem
\ref{thm:density}.
\end{proof}

In general, there is no hope to exactly match a Lagrangian metric
density by the push-forward/pull-back density of a localized tomograph
as in the proof of Theorem \ref{thm:symmetry}. However, to prove
Theorem \ref{thm:torus} it is sufficient to loosely bound the density
from below and this is done in Theorem \ref{thm:bound} below.

Let $L=\T^n\subset M^{2n}$ be a Lagrangian torus. Fix a compatible
metric on $M$ and denote by $\fg$ the metric density. We consider
Lagrangian tomographs $\CT$ with fiber $K$ diffeomorphic to $L$ and a
ball $B^d$ serving as the base $B$. Thus $E=L\times B$ and $\pi$ is
the projection to the second factor.

\begin{Theorem}
  \label{thm:bound}
  For any open set $U\supset L$ and any $\eta>0$, there exists a
  Lagrangian tomograph $\CT$ as above with $L=L_0=\Psi(\pi^{-1}(0))$,
  supported in $U$ and such that
\begin{itemize}
\item[\reflb{LT1}{\rm{(i)}}] $\fd_\CT|_L=\fg|_L$ pointwise,
\item[\reflb{LT2}{\rm{(ii)}}] $\fd_\CT\leq (1+\eta)\fg$.
\end{itemize}
\end{Theorem}

We prove this theorem in Section \ref{sec:loc-tom}.

\begin{proof}[Proof of Theorem \ref{thm:torus}]
 Set $\fd:=\fd_\CT$ for the sake of brevity. By Theorem
 \ref{thm:density} and since $\CT$ is supported in $U$, for any
 $\delta>0$,
  \begin{equation}
    \label{eq:tL-L-d}
  \int_{L\cap U}\fd=\int_{L}\fd\geq \int_{L_0}\fd-\delta
\end{equation}
when $d_\gamma(L_0,L)$ is small. Thus we have
\begin{equation*}
  \begin{aligned}
    \int_{L\cap U}\fg & \geq (1+\eta)^{-1} \int_{L\cap U}\fd
    & \textrm{by \ref{LT2}} \\
    & \geq (1+\eta)^{-1}\Big(\int_{L_0}\fd-\delta\Big)
    & \textrm{by \eqref{eq:tL-L-d}} \\
    & \geq (1+\eta)^{-1}\Big(\int_{L_0}\fg-\delta\Big) & \textrm{by
      \ref{LT1}}.
\end{aligned}
\end{equation*}  
Hence, for any $\eps>0$,
$$
  \int_{L\cap U}\fg \geq \int_{L_0}\fg -\eps
  $$
  once $\eta>0$ and $\delta>0$ and then $d_\gamma(L_0,L)$ are small
  enough.
\end{proof}

\section{Proof of Theorem \ref{thm:bound}}
\label{sec:loc-tom}

We carry out the proof in three steps. In the first step we discuss
some preliminaries, then in the second step we introduce the
Lagrangian tomographs which are used in the proof, and the last step
comprises the actual proof of Theorem~\ref{thm:bound}.

\subsubsection*{Step 1}

Let $\T^n=S^1\times \cdots \times S^1$ ($n$ times), where
$S^1=\R/2\pi\Z$, with angular coordinates $x=(x_1,\ldots,x_n)$ and let
$M=T^*\T^n=\T^n\times \R^n$ with coordinates $(x,y)$. We get a similar
decomposition of $T_{(x,y)}M=\R^n\times \R^n$ and we denote the
resulting coordinates on this space by $(w,u)$.

For $k \in \Z^+$, set
$$
\Pi_k\colon \T^n\to\T^n, \quad \Pi_k(x)=kx=(kx_1,\ldots,kx_n)
$$
and
$$
F_k=k(\Pi_k^*)^{-1}\colon \T^n\times\R^n\to\T^n\times\R^n, \quad
F_k(x,y)=(kx,y).
$$
We note here that $\Pi_k^*$ as a map from $T^*\T^n=\T^n\times \R^n$ to
itself is not defined, but its inverse $(\Pi_k^*)^{-1}$ is.

Let $\fd$ be an $n$-density on $M$. Consider the density
$$
\frac{1}{k^n}F^*_k\fd =:\fd_k.
$$
Explicitly, for an $n$-frame $\bar{v}=(v_1,\ldots, v_n)$ at $(x,y)$,
write $v_i=(w_i,u_i)\in \R^n\times \R^n=T_{(x,y)}M$, and
$$
\fd(\bar{v})=:h_{(x,y)}(w_1,u_1,\ldots,w_n,u_n).
$$
Then
$$
\fd_k(\bar{v})=\frac{1}{k^n}h_{(kx,y)}(k w_1,u_1,\ldots, k w_n, u_n)
=h_{(kx,y)}(w_1,u_1/k,\ldots, w_n, u_n/k).
$$
Hence, when $\fd$ is $\T^n$-invariant, i.e., $h$ is independent of
$x$, we have
$$
\fd_k(\bar{v})=h_{(x,y)}(w_1,u_1/k,\ldots, w_n, u_n/k).
$$
Setting $P_k(v):=(w,u/k)$ for $v=(w,u)$ and $P_k(\bar{v}):=(P_k
v_1\ldots, P_k v_n)$, we can rewrite the above expression as
$$
\fd_k(\bar{v})=\fd(P_k\bar{v})=:(P_k^*\fd)(\bar{v}).
$$
In other words, for any invariant density $\fd$,
\begin{equation}
  \label{eq:Pk-Fk}
  \frac{1}{k^n}F_k^*\fd= P_k^*\fd.
\end{equation}

In a similar vein, let $P_\infty$ be the projection to the horizontal
direction: $P_\infty(v)=(w,0)$ for $v=(w,u)$. We conclude that
\begin{equation}
  \label{eq:limit}
\fd_k=P^*_k\fd\to \fd_\infty:=P^*_\infty \fd
\end{equation}
uniformly on compact sets. In the next step, we explain how to obtain
such sequences $\fd_k$ from tomographs.

\subsubsection*{Step 2}

Let $\CT$ be the Lagrangian tomograph 
$$
\begin{tikzcd}
  (B, ds) & B \times \T^n \arrow[swap]{l}{\pi} \arrow{r}{\Psi} & \T^n
  \times \R^n,
\end{tikzcd}
$$
given by 
\begin{equation}
  \label{eq:CT}
  \Psi(s,x) = \big(x, \rho_1\sin(x_1+\phi_1), \ldots,
  \rho_n\sin(x_n+\phi_n)\big),
\end{equation}
where 
$$
s=\big((\rho_1,\phi_1),\ldots, (\rho_n,\phi_n)\big)
$$
is a point in the polydisk $B=(B^2)^n$ with $\rho_i\leq R$, for all
$i$ and some $R>0$. (Here we think of $(\rho_i,\phi_i)$ as polar
coordinates in the $i$-th copy of the disk $B^2$.) In other words,
$L_s = \Psi(s, \T^n)$ is the graph of the one-form
\begin{equation}
  \label{eq:sin}
\alpha_s=\sum_{i=1}^n\rho_i\sin(x_i+\phi_i) \,dx_i. 
\end{equation}
As the measure $ds$, we can take any smooth rotationally symmetric
(i.e., independent of $\phi_i$) measure supported away from the
boundary of $B$ and such that $\int_B ds >0$. This data gives an
"equivariant" tomograph, and, as a result, the pull-back/push-forward
density $\fd_{\CT}=\Psi_*\pi^*ds$ is $\T^n$-invariant.

More precisely, consider the $\T^n$-action on $B \times \T^n$ defined
by
\[
\theta (s, x) = ( (\rho, \phi -\theta), x+\theta )
\]
for $\theta \in \T^n$. Since $\Psi \theta = \theta \Psi$ and
$\theta_*\pi^*ds=\pi^*ds$, we also have
$\theta_*\Psi_*\pi^*ds=\Psi_*\pi^*ds$. Here the identity
$\theta_*\pi^*ds=\pi^*ds$ is a consequence of the assumption that $ds$
is rotationally symmetric. Namely, denote by $\bar{\theta}$ the
restriction of $\theta$ to $B$. Then, $\pi^*\bar{\theta}_*ds= \pi^*ds$
by the symmetry and $\theta_*\pi^*ds=\pi^*\bar{\theta}_*ds$ is easy to
see since $\theta$ and $\bar{\theta}$ are diffeomorphisms (see
\cite[Thm. 6.2]{APF07} for a more general statement).

Next, define $\CT_k$ to be the tomograph given by the data
$$
\begin{tikzcd}
  (B, 1/k^n ds)   & B \times \T^n \arrow[swap]{l}{\pi}
  \arrow{r}{\Psi_k} & \T^n \times \R^n,
\end{tikzcd}
$$
where
$$
\Psi_k(s,x)= \big(x, \rho_1\sin(kx_1+\phi_1), \ldots,
\rho_n\sin(kx_n+\phi_n)\big).
$$
In other words, the tomograph $\CT_k$ comprises the graphs of the
differential forms $\Pi_k^* \alpha_s/k$. Note that we scaled the
measure on $B$ by $1/k^n$.

\begin{Remark}
In the framework of Lemma \ref{lemma:imm}, $\CT$ corresponds to the
standard embedding $j\colon \T^n\hookrightarrow \R^{2n}$ and $\CT_k$
corresponds to the immersion $(j\circ \Pi_k)/k$.
\end{Remark}

\begin{Lemma}
  \label{lemma:push}
  We have
  $$
  \fd_{\CT_k}=P_k^*\fd_\CT.
  $$
\end{Lemma}

Combining this lemma with \eqref{eq:limit}, we observe that
$$
\fd_{\CT_k}\to P^*_\infty \fd_\CT
$$
uniformly on compact sets. Therefore, for any $\eta>0$,
\begin{equation}
  \label{eq:d-CTk}
\fd_{\CT_k}\leq (1+\eta) P^*_\infty \fd_\CT
\end{equation}
when $k$ is large enough, by \eqref{eq:density}.

Before turning to the proof of Lemma \ref{lemma:push}, we note another
feature of the tomograph $\fd_\CT$. Namely, by $\T^n$-invariance,
$$
P^*_\infty \fd_\CT =\sigma(y)\, |dx_1\wedge\ldots\wedge dx_n|
$$
for some smooth, compactly supported, non-negative function $\sigma$
on $\R^n$. The next lemma asserts that this function attains its
maximum at the origin.

\begin{Lemma}
\label{lemma:constant}
There exists  a constant $c>0$ such that 
\[
P^*_\infty \fd_\CT \leq c\, |dx_1\wedge\ldots\wedge dx_n|
\]
with equality along the zero section. In other words,
$c:=\sigma(0)=\max \sigma$. 
\end{Lemma}

\begin{proof}
Set $\T^n_y := \T^n \times \{ y\} \subset \T^n \times \R^n$. Then
\begin{equation}
  \label{eq:max-at-0}
\vert L_s \cap \T^n_y  \vert \leq \vert L_s \cap \T^n_0 \vert
\end{equation}
for all $y \in \R^n$ and $s \in B$. It follows that 
\begin{equation}
  \int_{\T^n_y} P^*_\infty \fd_\CT =\int_B | L_s \cap \T^n_y |\, ds \leq
  \int_B | L_s \cap \T^n_0 | \, ds  = \int_{\T^n_0} P^*_\infty \fd_\CT 
\end{equation}
for all $y \in \R^n$. We have $c:=\sigma(0) = \max_{y\in \R^n} \sigma(y)$. 
\end{proof}

The following observation we will be useful in the proof of Lemma
\ref{lemma:push}.

\begin{Lemma}
  \label{lemma:inverse}
  For any density $\fd$ on $\T^n \times \R^n$, we have
  $F_k^* {F_k}_* \fd = k \fd$. Moreover, if $\fd$ is
  $\T^n$-invariant, then ${F_k}_* F_k^* \fd = k \fd$.  
\end{Lemma}

\begin{proof}
  Let $\fd$ be an $m$-density on $M=\T^n \times \R^n$. As above, for
  an $m$-frame $\bar{v}=(v_1,\ldots, v_m)$ at $(x,y) \in M$, we write
  $v_i=(w_i,u_i)\in \R^n\times \R^n =T_{(x,y)}M$, and
$$
\fd(\bar{v})=:h_{(x,y)}(w_1,u_1,\ldots,w_m,u_m).
$$
Then
$$
{F_k}_* \fd (\bar{v})
= \sum_{i=0}^{k-1} h_{((x+i)/k,y)}(w_1/k,u_1,\ldots,w_m/k,u_m)
$$
and 
$$
F_k^* {F_k}_* \fd (\bar{v}) = \sum_{i=0}^{k-1}
h_{(x,y)}(w_1,u_1,\ldots,w_m,u_m) = k \fd(\bar{v}).
$$
Now suppose that $\fd$ is $\T^n$-invariant. Then 
$$
{F_k}_* F_k^* \fd (\bar{v})
= \sum_{i=0}^{k-1}  h_{(x+i/k,y)}(w_1,u_1,\ldots,w_m,u_m),
$$
which is equal to $k \fd(\bar{v})$ since $h$ is independent of $x$.
\end{proof}

Next, we prove Lemma \ref{lemma:push}. We will partially follow the
argument given in \cite[Thm. 6.2]{APF07}.

\begin{proof}[Proof of Lemma \ref{lemma:push}]
Consider the commutative diagram
\[
  \begin{tikzcd}
    B \arrow[swap]{d}{\id} & B \times \T^n \arrow[swap]{l}{\pi}
    \arrow{d}{\id \times \Pi_k} \arrow{r}{\Psi_k}
    & \T^n \times \R^n  \arrow{d}{F_k} \\
    B & B \times \T^n \arrow[swap]{l}{\pi} \arrow{r}{\Psi} & \T^n
    \times \R^n.
\end{tikzcd}
\]  
Observe that $\pi^*ds$ is $\T^n$-invariant and
$\pi^* ds= (\id \times \Pi_k)^* \pi^* ds$.  For the latter we used the
commutativity of the first block. Now one can deduce from the proof of
Lemma \ref{lemma:inverse} that
$$
(\id \times \Pi_k)_* \pi^* ds = k\pi^* ds.
$$
Also, the commutativity of the second block yields
$$
\Psi_*(\id \times \Pi_k)_* \pi^* ds = {F_k}_* {\Psi_k}_* \pi^* ds.
$$
By the above equalities and the definition of pull-back/push-forward
density, \eqref{eq:fd},
$$
k \fd_{\CT}=  {F_k}_* k^n \fd_{\CT_k}. 
$$
(Recall that in the tomograph $\CT_k$ the measure on $B$ is scaled by
$1/k^n$.) Applying $F_k^*$ to both sides and using Lemma
$\ref{lemma:inverse}$, we obtain
\[
F_k^* \fd_{\CT} = k^n \fd_{\CT_k}.
\]
On the other hand, since $\fd_{\CT}$ is $\T^n$-invariant, 
\[
F_k^* \fd_{\CT}/k^n = P_k^*\fd_\CT
\]
by \eqref{eq:Pk-Fk}. Thus
$$
\fd_{\CT_k}=P^*_k\fd_{\CT}.
$$
\end{proof}

\begin{Remark}
  The homogenization procedure described here is somewhat similar to
  the one from \cite{Vi:hom}, although we apply it to tomographs and
  densities rather than Hamiltonians, and the $\T^n$-invariance
  condition considerably simplifies the situation. (In the setting of
  that paper homogenization is trivial for invariant Hamiltonians.) We
  also note that Lemma \ref{lemma:push} holds for any equivariant
  tomograph $\CT$, when $\CT_k$ is defined as a tomograph satisfying
  the condition $F_k\circ \Psi_k=\Psi\circ (\id\times \Pi_k)$ with
  renormalized measure $ds/k^n$. It is easy to see that $\CT_k$
  exists, but it is not unique unless we require that
  $\Psi_{k}|_{\pi^{-1}(0)}=\id$. In any case, the density
  $\fd_{\CT_k}$ is independent of the choice of $\CT_k$. On the other
  hand, Lemma \ref{lemma:constant} relies on \eqref{eq:max-at-0} which
  is satisfied for the tomograph $\CT$ given by \eqref{eq:CT}, but not
  for an arbitrary equivariant tomograph comprising the graphs of
  exact forms. Finally, the reader has certainly noticed that while
  $\CT$ and $\CT_k$ are given by simple and explicit formulas, the
  proof is quite indirect. The reason is that we do not have an
  explicit and easy to work with expression for the densities
  $\fd_{\CT}$ and $\fd_{\CT_k}$, even for such simple
  tomographs. (Such an expression would depend on $ds$.)
  \end{Remark}

Finally, we are in a position to prove Theorem \ref{thm:bound}. 

\subsubsection*{Step 3}
Let now, as in the statement of the theorem, $M$ be the ambient
symplectic manifold equipped with a compatible metric and let $\fg$ be
the metric $n$-density. It is clear that in the proof of the theorem
we may replace $U$ by any open subset containing $L_0=\T^n$. Thus,
without loss of generality, we can identify $U$ with a neighborhood of
the zero section in $T^*\T^n$ so that the fibers are orthogonal to the
zero section; this Weinstein tubular neighborhood structure will be
used in what follows.  Fix some angular coordinates
$x=(x_1,\ldots,x_n)$ on $\T^n$ and further identify
$T^*\T^n=\T^n\times\R^n$ by using these coordinates as above.

Without loss of generality, by Moser's theorem, we can assume that
$$
\fg|_{\T^n}= |dx_1\wedge\ldots\wedge dx_n| 
$$
i.e., these two densities agree pointwise on the frames tangent to the
zero section.  Next, note that, along the zero section (but not
necessarily only on the frames tangent to the zero section),
$$
|dx_1\wedge\ldots\wedge dx_n|\leq \fg.
$$
To see this, consider the
standard Euclidean metric $n$-density on $\R^{2n}=\R^n\times\R^n$. Let
$\bar{v}$ be the image of the coordinate frame in $\R^n$ under a map
of the form
$$
(A,B)\colon \R^n\to \R^n\times\R^n.
$$
Then
$$
\fg(\bar{v})=\sqrt{\det(AA^*+BB^*)}\quad\textrm{(Pythagorean
  theorem).}
$$
Hence, 
$$
\fg(\bar{v})\geq\sqrt{\det(AA^*)}=|\det A|=|dx_1 \wedge\ldots\wedge
dx_n|(\bar{v}).
$$
We conclude by continuity of the metric density $\fg$ that for all
$\eta>0$ there exists a neighborhood $V \subset U$ of the zero section
such that on $V$ we have
$$
|dx_1\wedge\ldots\wedge dx_n| \leq (1+\eta) \fg.
$$
Strictly speaking, continuity only implies that the inequality holds
on a compact subset of the Stiefel bundle of $TV$. However, then it
holds for every $n$-frame tangent to $V$ since densities are
homogeneous or, to be more precise, by \eqref{eq:density}.

Let $\CT$ be the tomograph discussed in Step 2. Shrink the support of
$ds$ so that the density $\fd_\CT$ is supported in $V$. (Abusing
notation, here and below, as we modify $ds$ and hence $\CT$, we keep
the same notation.) Next, divide the measure $ds$ by the constant $c$
provided by Lemma \ref{lemma:constant} so that we have
$$
P^*_\infty \fd_\CT \leq  |dx_1\wedge\ldots\wedge dx_n|
$$
with equality along the zero section. Combining \eqref{eq:d-CTk} and
the previous two inequalities, we see that when $k$ is large
$$
\fd_{\CT_k}\leq (1+\eta) P_\infty^*\fd_{\CT} \leq (1+\eta)\,
|dx_1\wedge\ldots\wedge dx_n|\leq (1+\eta)^2\fg
$$
on $V$. Since $\fd_{\CT_k}$ vanishes outside $V$, we have
$\fd_{\CT_k}\leq (1+\eta)^2\fg$ on $U$. Note that $\fd_{\CT_k}=\fg$ on
$n$-frames tangent to the zero section as well. To complete the proof,
it remains to replace $\CT$ by $\CT_k$ and change $\eta$ so that
$(1+\eta)^2$ becomes $1+\eta$. \qed

\end{document}